\documentclass[10pt]{article}
\usepackage{amsmath,amssymb,amsthm,pb-diagram,lamsarrow,pb-lams, hyperref, euscript, ulem, mathcomp}
\usepackage[matrix,arrow,curve]{xy}
\usepackage{graphicx}
\usepackage{tabularx}
\usepackage{float}
\usepackage{hyperref}

\usepackage[T1]{fontenc}

\usepackage[sc]{mathpazo}
\usepackage[usenames]{color}
\usepackage{colortbl}

\DeclareFontFamily{T1}{pzc}{}
\DeclareFontShape{T1}{pzc}{m}{it}{1.8 <-> pzcmi8t}{}
\DeclareMathAlphabet{\mathpzc}{T1}{pzc}{m}{it}
\title{Universal covering space of the noncommutative torus}

\theoremstyle{plain}
\newtheorem{prop}{Proposition}[section]

\newtheorem{lem}[prop]{Lemma}%[section]
\newtheorem{thm}[prop]{Theorem}%[section]

\theoremstyle{definition}
\newtheorem{defn}[prop]{Definition}%[section]
\newtheorem{empt}[prop]{}%[section]
\theoremstyle{remark}

\chardef\bslash=`\\ % p. 424, TeXbook
\newcommand{\rar}{\rightarrow}

\newbox\ncintdbox \newbox\ncinttbox %% noncommutative integral symbols
\setbox0=\hbox{$-$} \setbox2=\hbox{$\displaystyle\int$}
\setbox\ncintdbox=\hbox{\rlap{\hbox
    to \wd2{\hskip-.125em \box2\relax\hfil}}\box0\kern.1em}
\setbox0=\hbox{$\vcenter{\hrule width 4pt}$}
\setbox2=\hbox{$\textstyle\int$} \setbox\ncinttbox=\hbox{\rlap{\hbox
    to \wd2{\hskip-.175em \box2\relax\hfil}}\box0\kern.1em}

\begin{document}
\maketitle  \setlength{\parindent}{0pt}
\begin{center}
\author{}
{\textbf{Petr R. Ivankov*}\\
e-mail: * monster.ivankov@gmail.com \\
}
\end{center}

\vspace{1 in}

\begin{abstract}
\noindent

Gelfand - Na\u{i}mark theorem supplies contravariant functor from a category of commutative $C^*-$  algebras to a category of locally compact Hausdorff spaces. Therefore any commutative $C^*-$ algebra is an alternative representation of a topological space. Similarly a category of  (noncommutative) $C^*-$ algebras can be regarded as a category of generalized (noncommutative) locally compact Hausdorff spaces. Generalizations of topological invariants may be defined by algebraic methods. For example Serre Swan theorem states that complex topological $K$ - theory coincides with $K$ - theory of $C^*$ - algebras. However the algebraic topology have a rich set of invariants. Some invariants do not have noncommutative generalizations yet. This article contains a sample of noncommutative universal covering. General theory of noncommutative universal coverings is being developed by the author of this article. However this sample has independent interest, it is very easy to understand and does not require knowledge of Hopf-Galois extensions.

\end{abstract}
\tableofcontents

\section{Introduction}

Following Gelfand-Na\u{i}mark theorem  \cite{murphy} states that category of locally compact Hausdorff topological spaces is equivalent to a category of commutative $C^*-$ algebras.

\begin{thm}\label{gelfand-naimark}
Let $\mathbf{Haus}$  be a category of locally compact Hausdorff spaces with continuous proper maps as morphisms. And, let $\mathbf{C^*Comm}$  be the category of commutative $C$ -algebras with proper *-homomorphisms (send approximate units into approximate units) as morphisms. There is a contravariant functor $C:\mathbf{Haus}\rar\mathbf{C^*Comm}$  which sends each locally compact Hausdorff space $X$ to the commutative $C^*$ -algebra $C_0(X)$ ($C(X)$ if $X$ is compact). Conversely, there is a contravariant functor $  \Omega: \mathbf{C^*Comm}\rar \mathbf{Haus}$  which sends each commutative $C^*$ -algebra $A$ to the space of characters on $A$ (with the Gelfand topology).

The functors $C$ and $\Omega$ are an equivalence of categories.
\end{thm}

So any (noncommutative) $C^*-$ algebra may be regarded as generalized (noncommutative)  locally compact Hausdorff topological space.
We may summarize several properties of the Gelfand Na\u{i}mark cofunctor with the
following dictionary.
\newline
\break
\begin{tabular}{|c|c|}
\hline
TOPOLOGY & ALGEBRA\\
\hline
Locally compact space & $C^*$ - algebra\\
Compact space & Unital $C^*$ - algebra\\
Continuous map & *-homomorpfism\\
Minimal compactification & Unitization\\
Maximal compactification & Algebra if multpicators\\
Closed subset & Ideal\\

Disjoint union of topological spaces ($\coprod X_{\iota} $) & Direct sum of (pro) - $C^*$ algebras ($\oplus A_{\iota}$) \\
Principal fibration & Hopf-Galois extension. \\
Universal covering & ? \\

\hline
\end{tabular}
\newline
\newline
\break

This article assumes elementary knowledge of following subjects.
\begin{enumerate}
\item Algebraic topology  \cite{spanier:at}.
\item $C^*-$ algebras and operator theory \cite{blackadar:ko},  \cite{murphy},  

\end{enumerate}

We use following notation.
\newline
\begin{tabular}{|c|c|}
\hline
Symbol & Meaning\\
\hline
$\mathbb{N}$ & monoid of natural numbers \\
$\mathbb{Z}$ & ring of integers \\
$\mathbb{R}$ (resp. $\mathbb{C}$)  & Field of real (resp. complex) numbers \\
$\mathbb{Q}$  & Field of rational numbers \\
$H$ &Hilbert space \\
$\mathcal{B}(H)$ & Algebra of bounded operators on Hilbert space $H$\\
$\mathcal{K}(H)$ or $\mathcal{K}$ & Algebra of compact operators on Hilbert space $H$\\
$U(H) \subset \mathcal{B}(H) $ & Group of unitary operators on Hilbert space $H$\\
$U(A) \in A $ & Group of unitary operators of algebra $A$\\
$A^+$  & $C^*-$ algebra $A$ with adjointed identity\\
$M(A)$  & A multiplier algebra of $C^*$-algebra $A$\\
$C(X)$ & $C^*$ - algebra of continuous complex valued \\
 & functions on topological space $X$\\
$C_0(X)$ & $C^*$ - algebra of continuous complex valued \\
 & functions on topological space which tends to 0 at infinity\\
  $C_c(X)$ & Algebra of continuous functions with compact support  \\
  $\mathrm{sp}(a)$ & Spectrum of element of $C^*$-algebra $a\in A$  \\

\hline
\end{tabular}
\newline
\newline

If $\tilde{X} \rightarrow X$ is an universal covering then there is a natural $*$-homomorphism $f:C_0(X) \rightarrow M(C_0(\tilde{X}))$. Homomorphism $f$ can be defined by a faithful representations $\pi: C_0(X) \rightarrow B(H)$, $\pi^*: C_0(\tilde{X}) \rightarrow B(H)$ such that

\begin{equation}\label{repr_equ}
\pi^*(f(x)\tilde{x})= \pi(x)\pi^*(\tilde{x}), \ \pi^*(\tilde{x} f(x))=\pi^*(\tilde{x}) \pi(x), \ x\in C_0(X), \ \tilde{x} \in C_0(\tilde{X}).
\end{equation} 

We would like construct an analogue of \ref{repr_equ} for a noncommutative torus.

\section{Algebraic construction of known universal covering spaces}
 
\subsection{Algebraic construction of the $\mathbb{R}\rightarrow S^1$ covering}\label{r_arrow_s1}

Our construction contains two ingredients:
\begin{enumerate}
\item Algebraic analogue of $n$ - listed covering projection $f_n: S^1\rightarrow S^1$ ;
\item Algebraic analogue of $\mathbb{R}\rightarrow S^1$;
\end{enumerate}
\begin{empt} {\it Construction of $n$ - listed covering projection.}
It is well known that $C(S^1)$ is a $C^*$-algebra which is generated by a single unitary element $u\in U(C(S^1))$. The $C(S^1)$ algebra can be faithfully represented, i.e. there is an inclusion $C(S^1)\rightarrow B(H)$. Let $\mathrm{sp}(u) \in \mathbb{C}$ be the spectrum of the $u$ (it is known that $\mathrm{sp}(u)=\{z \in \mathbb{C}\ | \ |z|=1 \}$), $\phi \in \ B_{\infty}(\mathrm{sp}(u))$ is a Borel-measurable function such that 
\begin{equation}\label{root_n_eqn}
(\phi(z))^n = z \ (\forall z \in \mathrm{sp}(u)). 
\end{equation}
According to spectral theorem \cite{murphy} there exist $v = \phi(u) \in U(B(H))$ and $v^n = u$. Let $C(u) \rightarrow B(H)$ (resp. $C(v)\rightarrow B(H))$ be a $C^*$- algebra generated by $u$ (resp. $v$), then we have an inclusion $C(u) \subset C(v)$ which  corresponds to an $n$ - listed covering projection $f_n: S^1\rightarrow S^1$.
\end{empt}
\begin{empt}{\it Construction of $\mathbb{R}\rightarrow S^1$}.
A circle $S^1$ can be parameterized by an angle parameter $\theta \in [-\pi, \pi]$.
Any function $g \in C([-1, 1])$ such that $g(-1)=g(1)$ corresponds to $\varphi_g \in C(S^1)$ such that $\phi_g(\theta)=g(\theta/\pi)$.  Let us fix a sequence of operators $u_0 = u, u_1, u_2, ... \in B(H)$ such that $u_{n+1}^2 = u_n$, we have a sequence of inclusions $C(u)=C(u_0)\subset C(u_1) \subset C(u_2) \subset ... \rightarrow B(H)$. We would like to prove that this sequence and any representation $\pi: C(u) \rightarrow B(H)$ naturally defines a representation $\pi^*: C_0(\mathbb{R}) \rightarrow B(H)$. Let $C_c(\mathbb{R}) \subset C_0(\mathbb{R})$ be an algebra of functions with compact support. There is a natural inclusion $i: C_c(\mathbb{R}) \rightarrow B(H)$ defined by following way. If $f \in C_c(\mathbb{R})$ then there is $n \in \mathbb{N}$ such that support of $f$ is contained in $[-2^n, 2^n]$. If $f^*\in C([-1,1])$ is such that $f(x)=f^*(2^nx)$ then $f^*(-1)=f^*(1)=0$, and we can define $\phi_{f^*} \in C(S^1)$. If $\pi_n : C(S^1) \approx C(u_n) \rightarrow B(H)$ then we set $i(f)=\pi_n(\phi_{f^*})$. It is clear that this definition does not depend on $n$. Since   $C_c(\mathbb{R})$ is dense in $C_0(\mathbb{R})$,  an inclusion $i:C_c(\mathbb{R}) \rightarrow B(H)$ can be continuously extended to a representation $\pi^*:C_0(\mathbb{R}) \rightarrow B(H)$. Representations $\pi$ and $\pi^*$ satisfy (\ref{repr_equ}).
\end{empt}

\subsection{Algebraic construction of the $\mathbb{R}^2\rightarrow S^1 \times S^1$ covering}\label{comm_torus_uni_space}

\begin{empt} {\it Construction of $C^*$ - algebra}.
The $C(S^1 \times S^1)$ is generated by two unitary elements $u, v \in U(C(S^1 \times S^1))$. There is a faithful representation $C(S^1 \times S^1)\rightarrow B(H)$. This representation induces two faithful representations $\pi_u: C(u) \rightarrow \ B(H)$, \ $\pi_v: C(v) \rightarrow \ B(H)$.  Construction from section \ref{r_arrow_s1} supplies following two representations:
\begin{enumerate}
\item $\pi^*_u: C_0(\mathbb{R}) \rightarrow B(H)$,
\item $\pi^*_v: C_0(\mathbb{R}) \rightarrow B(H)$.
\end{enumerate}
It is naturally to suppose that $C_0(\mathbb{R}^2)$ is isomorphic to the norm completion of subalgebra of $B(H)$ generated by operators of following type: 
\begin{equation}\label{torus}
\pi^*_u(f_1)\pi^*_v(f_2), \ \pi^*_v(f_1)\pi^*_u(f_2); \ (f_1, f_2 \in C_0(\mathbb{R})).
\end{equation}
However it is not always true. This construction is not unique because there are different Borel-measurable functions which satisfy (\ref{root_n_eqn}). This algebra is not always commutative, because one can select element $u_1\in B(H)$ such that $u = u_1^2$ and $u_1v = - vu_1$. However any algebra constructed by (\ref{torus}) is representative of an unique Morita equivalence class.
\end{empt}
\begin{empt} {\it Morita equivalence.}
Although constructed above algebra is not always isomorphic to $C_0(\mathbb{R}^2)$ it is strongly Morita equivalent to it \cite{blackadar:ko}. As it is proven in \cite{brown_green_rieffel:morita_stable} a $\sigma$-unital $C^*$-algebra $A$ is strongly Morita equivalent to a $\sigma$-unital $C^*$-algebra  $B$ if there is a $*$-isomorphism $A\otimes \mathcal{K} \approx B\otimes \mathcal{K}$. I find that good noncommutative theory of universal coverings should be invariant with respect to Morita equivalence. This theory can replace $C^*$-algebras with their stabilizations (recall that the stabilization of a $C^*$ algebra $A$ is a $C^*$-algebra $A\otimes \mathcal{K}$).
\begin{defn}
Let $A$ be a $C^*$-algebra, $A\rightarrow B(H)$ is a faithful representation, $u \in U(A^+)$, $v \in U(B(H))$,  is such that $v^n=u$ and $v^i \notin U(A^+)$, ($i=1,..., n-1$). A {\it generated by $v$ algebra} is a minimal subalgebra of $B(H)$ which contains following operators: 
\begin{enumerate}
\item $v^i a; \ (a \in A, \ i=0, ..., n-1)$
\item $a v^i$.
\end{enumerate}
Denote by $A\{v\}$ a generated by $v$ algebra.
\end{defn}
\begin{lem} Let $A$ be a $C^*$-algebra, $A\rightarrow B(H)$ is a faithful representation,  $u\in U(A^+)$ is an unitary element such that $\mathrm{sp}(u)=\{z \in \mathbb{C} \ | \ |z| = 1 \}$, $\xi, \eta \in B_{\infty}(\mathrm{sp}(u))$ are Borel measured functions such that $\xi(z)^n = \eta(z)^n = z$ ($\forall z \in\mathrm{sp}(u)$). Then there is an isomorphism 
\begin{equation}
A\{\xi(u)\} \otimes \mathcal{K} \rightarrow A\{\eta(u)\} \otimes \mathcal{K} 
\end{equation}
which is also a left $A$-module isomorphism. The isomorphism is given by
\begin{equation}
\xi(u) \otimes x \mapsto \eta(u) \otimes \xi\eta^{-1}(u)x; \ (x \in \mathcal{K}).
\end{equation}
\end{lem}
\begin{proof}
Follows from the equality $\xi(u)= \xi\eta^{-1}(\eta(u))$.
\end{proof}
Let $u\in U(C(S^1 \times S^1))$ be an unitary such that $u \neq v^n$ for any $n > 1, v \in U(C(S^1 \times S^1))$. One can construct different sequences $x_0 = u, x_1, x_2, ... \in B(H)$, $y_0 = u, y_1, y_2, ... \in B(H)$, such that $x_{n+1} = x_n^2$, $y_{n+1} = y_n^2$ but $C^*$ - algebras
\begin{equation}\nonumber
A\{x_1\}\subset A\{x_2\}\subset ...
\end{equation}
are not isomorphic to $C^*$ - algebras
\begin{equation}\nonumber
A\{y_1\}\subset A\{y_2\}\subset ... \ .
\end{equation}
However following sequences 
\begin{equation}\nonumber
A\{x_1\} \otimes \mathcal{K} \subset A\{x_2\} \otimes \mathcal{K}\subset ...
\end{equation}
\begin{equation}\nonumber
A\{y_1\} \otimes \mathcal{K} \subset A\{y_2\} \otimes \mathcal{K}\subset ... \ 
\end{equation}
contain isomorphic algebras.
If $A$ is the norm completion of an algebra generated by (\ref{torus}) then $A \otimes \mathcal{K} \approx C_0(\mathrm{R}^2) \otimes \mathcal{K} $, i.e. $A$ is strongly Morita equivalent to $C_0(\mathrm{R}^2)$. 

\end{empt}

\section{Universal covering of a noncommutative torus}

A noncommmutative torus \cite{varilly:noncom} $A_{\theta}$ is a $C^*$-algebra generated by two unitary elements ($u, v \in U(A_{\theta})$) such that
\begin{equation}\nonumber
uv = e^{2\pi i\theta}vu, \ (\theta \in \mathbb{R}).
\end{equation}
If $\theta\in \mathbb{Q}$ then noncommutative torus is strongly Morita equivalent to commutative one, i.e. $A_{\theta}\otimes \mathcal{K} \approx C(S^1 \times S^1) \otimes \mathcal{K}$,  and our construction is the same as in the section \ref{comm_torus_uni_space}. A case $\theta \notin \mathbb{Q}$ is more interesting. However a construction universal covering fully coincides with the considered in section \ref{comm_torus_uni_space} one. Let $A_{\theta}\rightarrow B(H)$ be a faithful representation. This representation induces two representations $\pi_u: C(u) \rightarrow B(H)$, $\pi_v: C(v) \rightarrow B(H)$. These representations induce representations $\pi^*_u(C_0(\mathbb{R})) \rightarrow B(H)$, $\pi^*_v(C_0(\mathbb{R})) \rightarrow B(H)$. The universal algebra of noncommutative torus is a norm completion of an algebra generated by  operators of following type
\begin{equation}\nonumber
\pi^*_u(f_1)\pi^*_v(f_2), \ \pi^*_v(f_1)\pi^*_u(f_2); \ (f_1, f_2 \in C_0(\mathbb{R})).
\end{equation}

This algebra is not unique but it is a representative of the unique strong Morita equivalence class.

\section{Discussion}

It is known the Maxwell's equations of classical electrodynamic are more important than Maxwell's proof. Now I am occupied by general theory of noncommutative universal coverings, which uses theory of Hopf $C^*$-algebras and Hopf-Galois extensions \cite{hajac:toknotes}. However I obtained a new algebra which can be regarded as a locally compact spectral triple \cite{varilly:noncom}. Maybe this result is more interesting than a general theory. This algebra is also interesting because it contains almost commutative sector, i.e. noncommutativity parameter $\theta$ is infinitesimal. It means that there is a sequence of algebras
\begin{equation}\nonumber
A_{\theta} \rightarrow A_{\theta / 2} \rightarrow ...
\end{equation}
such that noncommutativity parameter tends to 0. Maybe this algebra has a physical sense. I found an analogy of this algebra with \cite{kaluza_klein_gravity} Kaluza-Klein theory. In Kaluza-Klein theory we do not observe compact dimensions of the Universe because they are compact. We observe flat quotient space, which corresponds to a subalgebra of the Universe. Maybe we observe almost commutative subalgebra of the Universe, because we cannot observe a noncommutative algebra.

\section{Acknowledgment}
I would like to acknowledge a "Non-commutative geometry and topology" seminar, organized by:
\begin{enumerate}
\item Prof. Alexander Mishchenko,
\item Prof. Ivan Babenko,
\item Prof. Evgenij Troitsky,
\item Prof. Vladimir Manuilov,
\item Dr. Anvar Irmatov
\end{enumerate}
for discussion of my work.

\end{document}